\documentclass[10pt]{amsart}
\usepackage{amsmath,amssymb,amsthm,enumerate,bbm}
\allowdisplaybreaks
\usepackage[nameinlink,capitalize]{cleveref}
\numberwithin{equation}{section}

\newtheorem{theo}{Theorem}[section]
\newtheorem{coro}[theo]{Corollary}
\newtheorem{lemm}[theo]{Lemma}

\newtheorem{exam}[theo]{Example}
\newtheorem{assu}[theo]{Assumption}

\newcommand{\coref}[1]{c_{\eqref{#1}}}
\newcommand{\sptext}[3]{\hspace{#1 em}\mbox{#2}\hspace{#3 em}}

\newcommand{\od}{\mathrm{d}}
\newcommand{\supp}{{\rm supp}}

\newcommand{\B}{{\mathbb{B}}}
\newcommand{\D}{{\mathbb{D}}}
\newcommand{\E}{{\mathbb{E}}}
\renewcommand{\P}{{\mathbb{P}}}
\newcommand{\R}{{\mathbb{R}}}

\newcommand{\cI}{\mathcal I}
\newcommand{\cF}{{\mathcal{F}}}

\newcommand{\vare}{\varepsilon}

\newcommand{\BMO}{\mathrm{BMO}}
\newcommand{\Lip}{\rm Lip}
\newcommand{\hoelO}[1]{{\mbox{\rm H\"ol}}^0_{#1}(\R)}
\newcommand{\hoeldO}[1]{{\mbox{\rm H\"ol}}^0_{#1}(\R^d)}

\begin{document}

\title[On sharp embeddings]{On the sharpness of embeddings of H\"older spaces into Gaussian Besov spaces}
	
\author{Stefan Geiss}
\address{Department of Mathematics and Statistics, P.O.Box 35, FI-40014 University of Jyv\"as\-kyl\"a, Finland}
\email{stefan.geiss@jyu.fi}

\thanks{The author was supported by the Project 298641 \textit{'Stochastic Analysis and Nonlinear Partial Differential Equations,
        Interactions and Applications'} of the Academy of Finland.}        



\begin{abstract}
For an interpolation pair $(E_0,E_1)$ of Banach spaces with $E_1 \hookrightarrow E_0$  we use vectors 
$b_1,b_2,\ldots \in E_1$ that satisfy an extremal property with respect to the $J$- and $K$-functional to construct
sub-spaces that are isometric to $\ell_q^{(\theta)}$. The construction is based on a randomisation using
independent Rademacher variables. We verify that systems obtained by re-scaling a function with a certain periodicity 
property share this extreme property. This implies the sharpness of natural embeddings of H\"older spaces obtained by 
the real interpolation into the corresponding Gaussian Besov spaces.
\end{abstract}
\maketitle
\today


\section{Introduction}

Let us start with an abstract problem:
Assume two interpolation pairs $(F_0,F_1)$ and $(E_0,E_1)$ and a linear operator $T:F_0+F_1\to E_0+E_1$ such that 
$T:F_i\to E_i$ are continuous for $i=0,1$. Find sufficient conditions such that the embedding $T:(F_0,F_1)\to (E_0,E_1)$ is {\em sharp with 
respect to the real interpolation method} in the sense that  
\[ T((F_0,F_1)_{\theta,q_1}) \not \subseteq  (E_0,E_1)_{\theta,q_0} \]
for all $\theta \in (0,1)$ and $1 \le q_0 < q_1 \le \infty$.
A first typical example that does not share this type of sharpness is obtained by the re-iteration theorem:
\smallskip

\begin{exam}\label{statement:index_shift_by_re-iteration}
\rm
Assume that $F_1 \hookrightarrow F_0$ and let $E_0 := (F_0,F_1)_{\eta,p}$ and $E_1:=F_1$, where $(\eta,p)\in (0,1)\times [1,\infty]$ 
is fixed. Then one has $E_0+E_1=(F_0,F_1)_{\eta,p}$ and $F_0+F_1= F_0$. Let $T:E_0+E_1\to F_0+F_1$ be the natural embedding.
By the re-iteration theorem we have for $(\theta,q)\in (0,1)\times [1,\infty]$  that
\[ (E_0,E_1)_{\theta,q} = ((F_0,F_1)_{\eta,p},F_1)_{\theta,q} 
                        = (F_0,F_1)_{(1-\theta)\eta+ \theta,q}. \]
This implies that 
\[ T( (E_0,E_1)_{\theta,\infty}) = (F_0,F_1)_{(1-\theta)\eta+ \theta,\infty}
\subseteq (F_0,F_1)_{\theta,1} \]
with $(1-\theta)\eta+ \theta > \theta$.
\end{exam}
\smallskip

\cref{statement:index_shift_by_re-iteration} relies on the fact that the main interpolation parameter $\theta$ is already shifted.
An example that relies only on the shift of the fine-tuning parameter $q$ is the following:

\begin{exam}\rm
We let $a=(\alpha_n)_{n=1}^\infty \in \ell_r$ for some fixed $r\in [1,\infty)$.
Then we get a multiplier on the spaces $\ell_q^{(s)}$, that are recalled in \eqref{eqn:definition_ellp(s)}, by
\[ M_a : \ell_{q_i}^{(i)} \to  \ell_{p_i}^{(i)}
   \sptext{1}{defined by}{1}
   M_a ((x_n)_{n=1}^\infty) := (\alpha_n x_n)_{n=1}^\infty \]
for $i=0,1$ and $\frac{1}{p_i} = \frac{1}{q_i}+\frac{1}{r}$, 
$1 \le p_i \le q_i,r \le \infty$. According to \eqref{eqn:interpolation_ellp(s)}
real interpolation with parameters $(\theta,q)\in (0,1)\times [1,\infty]$ gives 
\[  M_a : \ell_q^{(\theta)} \to  \ell_q^{(\theta)}. \]
Choosing $\alpha_n := n^{-\kappa}$ with $\kappa>1/r$, one does not have a continuous map
\[  M_a : \ell_1^{(\theta)} \to  \ell_\infty^{(\eta)} \]
for any $0<\theta < \eta <1$. But still, one has the improvement on the fine-tuning level
\[  M_a : \ell_q^{(\theta)} \to  \ell_p^{(\theta)}
    \sptext{1}{with}{1} 1/p=1/q + 1/r, \]
which implies $p<q$ and corresponds to the map
\[  M_a : (\ell_{q_0}^{(0)},\ell_{q_1}^{(1)})_{\theta,q} \to (\ell_{q_0}^{(0)},\ell_{q_1}^{(1)})_{\theta,p}. \]
\end{exam}
\bigskip

As an application of our general result \cref{statement:random_basis_of_extreme_elements} we discuss the sharpness of the embedding 
\[ (C_b^0(\R^d),\Lip^0(\R^d))_{\theta,q} \hookrightarrow (L_2(\R^d,\gamma_d),\D_{1,2}(\R^d,\gamma_d))_{\theta,q}, \]
where $\gamma_d$ is the standard Gaussian measure on $\R^d$.
Starting with $d=1$ a typical example of a H\"older continuous function, $f(x):= (x^+)^\theta$, one realizes that 
\[ f\in \D_{1,2}(\R,\gamma_1) 
   \sptext{1}{if}{1} \theta \in \left (\frac{1}{2}, 1\right ], \]
so that this type of example does not work for us. 
The reason is that the H\"older space is much more sensitive to the singularity at $x=0$ than the Gaussian Besov space is. 
At least there are  two constructions in the literature to overcome this obstacle.
Firstly, a problem discussed in \cite{Laukkarinen:20} within the context of L\'evy processes, which is not directly related to ours, 
gives a hint how to overcome this obstacle: Instead of taking  $x\mapsto (x^+)^\theta$ one should use appropriate 
combinations of Schauder functions that have more singularities which can be seen by the Gaussian Besov spaces. 
Secondly, there exist characterisations for  Besov- and Triebel-Lizorkin-spaces, $B_{p,q}^s(\R^d,w)$ and $F_{p,q}^s(\R^d,w)$, where 
$w$ is a polynominal weight,  that are based on Wavelet constructions \cite{Haroske:Triebel:05}. 
On the one hand, our weight is Gaussian and therefore does not fall into the setting of polynomial weights.
Secondly, although our construction  in \cref{sec:application_Gaussian_Besov} is based on the principal idea of wavelets,
we do not need  the property that we have a basis so that only very mild assumptions on our construction are required.

Our main contribution is that we lift our arguments to a  construction for interpolation spaces between general Banach spaces,
where we choose elements sharing an extremal relation regarding the $J$- and $K$-functional and a  
randomisation. This construction might be useful for other problems too.


\section{Preliminaries}


\subsection{Spaces}
Given a function $f:\R^d\to \R$, we denote by $|f|_{\rm Lip}:= \sup_{x\not = y} |f(x)-f(y)|/|x-y|$ its Lipschitz constant, where
$|x|$ is the  euclidean norm  on $\R^d$.
We denote by $\Lip^0(\R^d)$ the space of all Lipschitz functions $f:\R^d\to \R$ such that $f(0)=0$  which becomes a normed 
space under $| \cdot |_{\Lip}$. The space of continuous bounded functions $f:\R^d\to \R$ vanishing at zero is denoted by $C_b^0(\R^d)$ 
and is equipped with the norm $\| f\|_{C_b^0(\R^d)}:= \sup_{x\in \R^d} |f(x)|$. 
For $(s,q)\in \R\times [1,\infty]$ we use the Banach space
\begin{equation}\label{eqn:definition_ellp(s)}
 \ell_q^{(s)} := \{ x=(x_n)_{n=1}^\infty \subseteq \R : \| x\|_{\ell_q^{(s)}} := \| (2^{sn}x_n)_{n=1}^\infty \|_{\ell_q}<\infty\}. 
\end{equation}
In this note we only consider real Banach spaces. Given Banach spaces $E_0,E_1$, the notation $E_1 \hookrightarrow E_0$ stands for an 
injective continuous embedding, so that there is a $c>0$ such that $\| x\|_{E_0} \le c \|x\|_{E_1}$ for all $x\in E_1$.
Moreover, we use the notation $A\sim_c B$, where $A,B\ge 0$ and $c\ge 1$, if $(1/c) A \le B \le c A$.


\subsection{Real interpolation}
Let $(E_0,E_1)$ be a couple of Banach spaces such that $E_0$ and $E_1$ are continuously embedded into 
some topological Hausdorff space $X$. 
We equip $E_0+E_1:= \{ x = x_0+x_1 : x_i \in E_i\}$ with the norm 
$\|x\|_{E_0+E_1} := \inf \{  \|x_0\|_{E_0}+\|x_1\|_{E_1}  : x_i\in E_i, x=x_0+x_1\}$
and $E_0 \cap E_1$ with the norm $\|x\|_{E_0\cap E_1} := \max \{ \|x\|_{E_0}, \|x\|_{E_1}\}$ to get Banach spaces 
$E_0\cap E_1 \subseteq E_0+E_1$.
For $x\in E_0\cap E_1$ and $x\in E_0+E_1$, respectively, and $\lambda \in (0,\infty)$ we define the $J$-and $K$-functional
\begin{align*}
J(\lambda,x;E_0,E_1) & := \max \{ \|x\|_{E_0},\lambda \|x\|_{E_1} \},\\
K(\lambda,x;E_0,E_1) & := \inf \{ \|x_0\|_{E_0} + \lambda \|x_1\|_{E_1} : x = x_0+x_1 \}.
\end{align*}
Given $(\theta,q)\in (0,1)\times [1,\infty]$ we obtain real interpolation spaces by
\begin{align*}
(E_0,E_1)_{\theta,q,J}  & := \left \{ x \in E_0+E_1 :  \|x\|_{(E_0,E_1)_{\theta,q,J}} < \infty  \right \},\\
(E_0,E_1)_{\theta,q,K}  & := \left \{ x \in E_0+E_1 :  \|x\|_{(E_0,E_1)_{\theta,q,K}} < \infty  \right \},
\end{align*}
with
\begin{align*}
     \|x\|_{(E_0,E_1)_{\theta,q,J}} 
& := \inf \Big \{  \left \| (2^{-k\theta} J(2^k,x_k;E_0,E_1))_{k=-\infty}^\infty \right \|_{\ell_q} : \\ 
& \hspace*{10em} 
     x = \sum_{k=-\infty}^\infty x_k \mbox{ in } E_0+E_1, x_k \in E_0\cap E_1 \Big \},\\
     \|x\|_{(E_0,E_1)_{\theta,q,K}} 
& := \left \| \lambda \mapsto  \lambda^{-\theta} K(\lambda,x;E_0,E_1) 
     \right \|_{L_q\left ((0,\infty), \frac{\od \lambda}{\lambda}\right )}.
\end{align*}
By \cite[Lemma 3.2.3, Theorem 3.3.1]{Bergh:Loefstroem:76} for all $(\theta,q)\in (0,1)\times [1,\infty]$ there is a 
constant $\coref{eqn:J_K_equivalence}=c(\theta,q) \ge 1$ such that
\begin{equation}\label{eqn:J_K_equivalence}
   \frac{1}{\coref{eqn:J_K_equivalence}} \|x\|_{(E_0,E_1)_{\theta,q,K}} \le \|x\|_{(E_0,E_1)_{\theta,q,J}} \le \coref{eqn:J_K_equivalence} \|x\|_{(E_0,E_1)_{\theta,q,K}}.
\end{equation}
Therefore one defines the two-parameter scale of real interpolation spaces 
\[ (E_0,E_1)_{\theta,q}:=(E_0,E_1)_{\theta,q,K}=(E_0,E_1)_{\theta,q,J}\]
and equip $(E_0,E_1)_{\theta,q}$ with the norm $\|\cdot\|_{(E_0,E_1)_{\theta,q}}:= \|\cdot\|_{(E_0,E_1)_{\theta,q,K}}$.
We obtain a family of Banach spaces $\left (  (E_0,E_1)_{\theta,q},\|\cdot\|_{(E_0,E_1)_{\theta,q}} \right )$ with the  lexicographical 
ordering
\begin{equation}\label{eqn:embedding_interpolation_same_theta}
(E_0,E_1)_{\theta,q_0} \subseteq (E_0,E_1)_{\theta,q_1} \sptext{1}{for all}{1} \theta \in (0,1)
\sptext{.5}{and}{.5} 1\leqslant  q_0 < q_1 \leqslant \infty, 
\end{equation}
and, under the additional assumption that $E_1 \hookrightarrow E_0$,  
\[ (E_0,E_1)_{\theta_0,q_0} \subseteq (E_0,E_1)_{\theta_1,q_1}
   \sptext{1}{for all}{1}
   0< \theta_1 < \theta_0<1 \sptext{.5}{and}{.5} q_0,q_1\in [1,\infty]. \]
For more information about interpolation theory the reader is referred to 
\cite{Bennet:Sharpley:88,Bergh:Loefstroem:76,Triebel:78}.


\subsection{Examples of interpolation spaces.}
For  $s_0 \not = s_1$ and $q_0,q_1,q\in [1,\infty]$ the spaces $\ell_q^{(s)}$ interpolate as
\begin{equation}\label{eqn:interpolation_ellp(s)}
 (\ell_{p_0}^{(s_0)},\ell_{p_1}^{(s_1)})_{\theta,q}
   = \ell_q^{(s)} 
   \sptext{1}{with}{1}
   s = (1-\theta) s_0 + \theta s_1,
\end{equation}
where the norms are equivalent up to multiplicative constants depending at most
on $s_0,s_1,q_0,q_1,q,\theta$ (see \cite[Theorem 5.6.1]{Bergh:Loefstroem:76}). Regarding the H\"older spaces, we use 
\[ \hoeldO{\theta,q} := (C_b^0(\R^d),\Lip^0(\R^d))_{\theta,q}. \]


\section{A general result}

Let $(M, \Sigma, \rho)$ be a probability space and $\varepsilon_n \colon M \to \{-1, 1\}$ be a sequence  of i.i.d. random variables with
$\rho(\varepsilon_n = \pm 1) = \frac{1}{2}$.

\begin{theo}
\label{statement:random_basis_of_extreme_elements}
Assume Banach spaces $(E_0,E_1)$ with $E_1\hookrightarrow E_0$, 
$(b_n)_{n=1}^\infty\subseteq E_1$, and $\delta,\kappa>0$ such that
\begin{equation}\label{eqn:J-K-extreme_points} 
\delta \le K(2^{-n},b_n;E_0,E_1) \le J(2^{-n}, b_n;E_0,E_1) \le \kappa
\end{equation}
for $n=1,2,\ldots$ Then for all $(\theta,q)\in (0,1)\times [1,\infty)$, there is a constant
$\coref{eqn:statement:random_basis_of_extreme_elements}=\coref{eqn:statement:random_basis_of_extreme_elements}(\theta,q,\kappa,\delta)>0$
such that for all $(\alpha_n)_{n=1}^\infty\subset \R $
such that $\# \{ n\ge 1 : \alpha_n \not = 0 \} < \infty$ one has
\begin{multline}\label{eqn:statement:random_basis_of_extreme_elements}
        \frac{1}{\coref{eqn:statement:random_basis_of_extreme_elements}} \left \| (\alpha_n)_{n=1}^\infty \right \|_{\ell_q^{(\theta)}} 
    \le  \int_M \left \|  \sum_{n=1}^\infty \alpha_n \vare_n(\xi) b_n  \right \|_{(E_0,E_1)_{\theta,q}} \rho(\od \xi)  \\
    \le  \sup_{\vare_n = \pm 1} \left \|  \sum_{n=1}^\infty \alpha_n \vare_n b_n \right \|_{(E_0,E_1)_{\theta,q}} 
    \le  \coref{eqn:statement:random_basis_of_extreme_elements}  \left \| (\alpha_n)_{n=1}^\infty \right \|_{\ell_q^{(\theta)}}.
\end{multline}
\end{theo}

\begin{proof}
Using \eqref{eqn:J_K_equivalence} one has 
\begin{align*}
     \left \| \sum_{n=1}^\infty \alpha_n b_n \right \|_{(E_0,E_1)_{\theta,q}}
&\le \coref{eqn:J_K_equivalence}  \| ( \alpha_n J(2^{-n},b_n; E_0,E_1))_{n=1}^\infty \|_{\ell_q^{(\theta)}} \\
&\le \coref{eqn:J_K_equivalence}  \kappa  \| ( \alpha_n)_{n=1}^\infty \|_{\ell_q^{(\theta)}}
\end{align*}
which proves the right-hand side inequality.
The left-hand side inequality is obtained with the help of the Khintchine-Kahane inequality
(with a constant $c_q \ge 1$ depending on $q$ only) by
\begin{align*}
&   \int_M \left \|  \sum_{n=1}^\infty \alpha_n \vare_n(\xi) b_n  \right \|_{(E_0,E_1)_{\theta,q}} \rho(\od \xi) \\
&\sim_{c_q} \left ( \int_M \left \|  \sum_{n=1}^\infty \alpha_n \vare_n(\xi) b_n  \right \|_{(E_0,E_1)_{\theta,q}}^q \rho(\od \xi) \right )^\frac{1}{q} \\
& =
   \left (  \int_0^\infty  \lambda^{-q\theta-1} \int_M K\left (\lambda,\sum_{n=1}^\infty \alpha_n \vare_n(\xi) b_n;E_0,E_1 \right )^q \rho(\od \xi)
   \od \lambda \right )^\frac{1}{q} \\
& \ge 
   \left ( \sum_{k=0}^\infty \int_{\frac{1}{2^{k+1}}}^{\frac{1}{2^k}} \lambda^{-q\theta-1} \int_M 
   K\left (\lambda,\sum_{n=1}^\infty \alpha_n \vare_n(\xi) b_n;E_0,E_1 \right )^q \rho(\od \xi)
   \od \lambda \right )^\frac{1}{q} \\
& \ge 
   \left ( \sum_{k=0}^\infty \int_{\frac{1}{2^{k+1}}}^{\frac{1}{2^k}} \lambda^{-q\theta-1} \int_M K\left (\frac{1}{2^{k+1}},\sum_{n=1}^\infty \alpha_n \vare_n(\xi) b_n;E_0,E_1 \right )^q \rho(\od \xi)
   \od \lambda \right )^\frac{1}{q} \\
& \ge 
   \left ( \sum_{k=0}^\infty \frac{1}{2^{k+1}} 2^{k(q\theta+1)} \int_M K\left (\frac{1}{2^{k+1}},\sum_{n=1}^\infty \alpha_n \vare_n(\xi) b_n;E_0,E_1 \right )^q \rho(\od \xi)
    \right )^\frac{1}{q} \\
& \ge 
   \left ( \frac{1}{2} \sum_{k=0}^\infty 2^{k q\theta } K\left (\frac{1}{2^{k+1}},\alpha_{k+1} b_{k+1};E_0,E_1 \right )^q 
    \right )^\frac{1}{q} \\
& \ge  \delta \left ( \frac{1}{2} \right )^\frac{1}{q} 
   \left ( \sum_{k=0}^\infty 2^{k q\theta } |\alpha_{k+1}|^q 
    \right )^\frac{1}{q} \\
& = \delta \left ( \frac{1}{2} \right )^\frac{1}{q} 2^{-\theta} \| (\alpha_n)_{n=1}^\infty \|_{\ell_q^{(\theta)}},
\end{align*} 
where we use
\[    \left ( \int_M K\left ( \frac{1}{2^{k+1}},\sum_{n=1}^\infty \alpha_n \varepsilon_n(\xi) b_n;E_0,E_1  
      \right )^q \rho(\od \xi) \right )^\frac{1}{q} 
  \ge K\left ( \frac{1}{2^{k+1}},\alpha_{k+1} b_{k+1};E_0,E_1  \right ).\]
This concludes the proof.
\end{proof}


\section{An application to Gaussian Besov spaces}
\label{sec:application_Gaussian_Besov}


\subsection{Gaussian Besov spaces}
We let $d\ge 1$ and let $\gamma_d$ be the standard Gaussian measure on $\R^d$. The space $L_2(\R^d,\gamma_d)$ is equipped 
with the orthonormal basis of generalised Hermite polynomials $(h_{k_1,...,k_d})_{k_1,...,k_d=0}^\infty$ 
given by
\[ h_{k_1,...,k_d}(x_1,...,x_d) := h_{k_1}(x_1)\cdots h_{k_d}(x_d), \]
where $(h_k)_{k=0}^\infty\subset L_2(\R,\gamma_1)$ is the orthonormal basis of Hermite polynomials. 
The Gaussian Sobolev space $\D_{1,2}(\R^d,\gamma_d)$ consists of all $f\in L_2 (\R^d,\gamma_d)$ such that
\[       \sum_{k_1,...,k_d=0}^\infty \langle f, h_{k_1,...,k_d} \rangle_{L_2(\R^d,\gamma_d)}^2
     \left\|  \nabla h_{k_1,...,k_d} \right \|_{L_2(\R^d,\gamma_d)}^2 
               <  \infty . \]
We equip $\D_{1,2}(\R^d,\gamma_d)$ with the norm 
\[ \| f \|_{\D_{1,2}(\R^d,\gamma_d)}:= \sqrt{\| f\|_{L_2(\R^d,\gamma_d)}^2 + \| Df \|_{L_2(\R^d,\gamma_d)}^2} \]
and obtain a Banach space.
For $(\theta,q) \in (0,1) \times [1,\infty]$ we let 
\[ \B_{2,q}^\theta(\R^d,\gamma_d):= (L_2(\R^d,\gamma_d),\D_{1,2}(\R^d,\gamma_d))_{\theta,q} \]
be the Gaussian Besov space of smoothness $\theta$ and with
fine-index $q$.


\subsection{Description of the $K$-functional}
\label{sec:description_K-functional}
Assume a basis $(\Omega,\cF,\P,(\cF_t)_{t\ge 0})$, such that 
$(\Omega,\cF,\P)$ is a complete probability space and $(\cF_t)_{t\ge 0}$ is the augmented natural 
filtration of a standard $d$-dimensional standard Brownian motion $(W_t)_{t\ge 0}$, 
where for convenience $W_0\equiv 0$ and all paths of $W$ are assumed to be continuous. 
Assume a corresponding copy $(\Omega',\cF',\P',(\cF'_t)_{t\ge 0},(W'_t)_{t\ge 0})$.
From \cite[Proposition 3.4]{Geiss:Toivola:15} we know that,
for $t\in (0,1)$,
\begin{multline} \label{eqn:K-functional_Gauss_space}
K(f,\sqrt{1-t};L_2(\R^d,\gamma_d),\D_{1,2}(\R^d,\gamma_d)) \\ 
\sim_{\coref{eqn:K-functional_Gauss_space}} \| f(W_1) - f(W_t+W'_{1-t}) \|_{L_2} + \sqrt{1-t} \| f \|_{L_2(\R^d,\gamma_d)} 
\end{multline}
where $\coref{eqn:K-functional_Gauss_space}>0$ is an absolute constant.


\subsection{A construction for condition  \eqref{eqn:J-K-extreme_points}}

Assuming Borel functions $b_n:\R^d\to \R$,  we replace condition \eqref{eqn:J-K-extreme_points}
as follows: there exist $\delta,\kappa >0$ such that, for all $n=1,2,\ldots$,
\begin{align} 
     \max\{ \| b_n \|_{L_2(\R^d,\gamma_d)}, 2^{-n} \| b_n \|_{\D_{1,2}(\R^d,\gamma_d)}  \}
&\le \kappa, \label{eqn:extreme_1} \\
     \| b_n(W_1) - b_n(W_{1-\frac{1}{4^n}} + W'_{\frac{1}{4^n}}) \|_{L_2} 
& \ge \delta \label{eqn:extreme_2}.
\end{align}
We use the following assumptions to construct such $b_n$:

\begin{assu}
\label{assertion:condition_mother_function}
\rm
We assume a Borel measurable $b:\R^d \to \R$ such that 
\begin{enumerate}
\item \label{eqn:1:assertion:condition_mother_function}
      $\sup_{x\in \R^d} |b(x)| \le \kappa$,
\item \label{eqn:2:assertion:condition_mother_function}
      $|b|_{\rm Lip} \le \kappa$,
\item \label{eqn:3:assertion:condition_mother_function}
      there is an $M>0$ such that 
      \[  \inf_{x\in \R} \int_{|y|\le M} \int_{|\overline{y}|\le M}  | b(x+ y)-  b(x+\overline{y}) |^2 \od y \od \overline{y} > 0.\]
\end{enumerate} 
\end{assu}

\begin{lemm}
\label{statement:scaling}
Assume the conditions in \cref{assertion:condition_mother_function} and define
\[ b_n(x) :=  b(2^{n-1}x)
   \sptext{1}{for}{1} n\ge 1. \]
Then \eqref{eqn:extreme_1} and \eqref{eqn:extreme_2} are satisfied.
\end{lemm}

\begin{proof}
Condition \eqref{eqn:extreme_1} follows from  assumptions 
\eqref{eqn:1:assertion:condition_mother_function} and \eqref{eqn:2:assertion:condition_mother_function}.
To verify condition \eqref{eqn:extreme_2}, we use
the scaling property, set $\sigma := \frac{1}{4}$, and get that
\begin{align*}
&    \left \| b_n(W_1) - b_n\left (W_{1-\frac{1}{4^n}} + W'_{\frac{1}{4^n}}\right ) \right \|_{L_2}^2 \\
& = \left \| b(W_{4^{n-1}}) - b\left ( W_{4^{n-1}-\frac{1}{4}} + W'_{\frac{1}{4}} \right ) \right \|_{L_2}^2 \\
&\ge \inf_{x\in \R^d} \| b(x + W_\sigma) - b(x + W'_\sigma)\|_{L_2}^2 \\
& \ge \left ( \frac{1}{\sqrt{2\pi \sigma^2}} e^{-\frac{M^2}{2 \delta }} \right )^d  
      \inf_{x\in \R^d} \int_{|y|\le M} \int_{|\overline{y}|\le M}  | b_1(x+ y)-  b_1(x+\overline{y}) |^2 \od y \od \overline{y} >0,
\end{align*}
which verifies \eqref{eqn:extreme_2}.
\end{proof}
\medskip

\begin{lemm}
\label{statement:periodicity_implies_oscillation}
Assume $b:\R^d \to \R$ to be continuous and bounded and an $R>0$ such that the following
is satisfied:
\begin{enumerate}[{\rm (1)}]
\item The function $b$ is not constant on $\{ x \in \R^d : |x| \le R\}$.
\item For all $x\in \R^d$ with $|x|>R$ there is $z\in \R^d$ with $|z| \le R$ such that
      \[ b(x+y) = b(z + y) \sptext{1}{for all}{1} y\in \R^d.\]
\end{enumerate}
Then condition \eqref{eqn:3:assertion:condition_mother_function} of \cref{assertion:condition_mother_function}
is satisfied.
\end{lemm}

\begin{proof}
We choose $M:=2R$ and get
\begin{multline*}
    \inf_{x\in \R^d} \int_{|y|\le M} \int_{|\overline{y}|\le M}  | b(x+ y)-  b(x+\overline{y}) |^2 \od y \od \overline{y}\\
  = \inf_{|z| \le R} \int_{|y|\le 2R} \int_{|\overline{y}|\le 2R}  | b(z+ y)-  b(z+\overline{y}) |^2 \od y \od \overline{y}.
\end{multline*}
Assuming the right-hand side to be zero implies by the continuity of $b$ that there is an $z_0\in \R^d$ with
$|z_0|\le R$ such that
\[ \int_{|y|\le 2R} \int_{|\overline{y}|\le 2R}  | b(z_0+ y)-  b(z_0+\overline{y}) |^2 \od y \od \overline{y} = 0. \]
Therefore, $b(z_0+y) = b(z_0)$ for all $|y| \le 2R$ and $b$ would be constant on 
$\{ x \in \R^d : |x| \le R\}$ which is a contradiction.
\end{proof}


\subsection{An application to H\"older spaces}

Now we apply \cref{statement:random_basis_of_extreme_elements} to the relation between H\"older functions and Gaussian Besov spaces.

\begin{coro}
\label{statement:sharp_embedding_Hoelder_Besov}
Let $\theta \in (0,1)$ and $1\le \underline{q}<\overline{q} \le \infty$.
Then
\[ \hoeldO{\theta,\overline{q}} \not\subseteq \B_{2,\underline{q}}(\R^d,\gamma_d).\]
\end{coro}

\begin{proof}
(a) Without loss of generality we can assume that $\overline{q} < \infty$.
We define the norms
\begin{align*}
\| f \|_0 &:= \max \left \{ \| f \|_{\hoeldO {\theta,\overline{q}}}, \| f\|_{\B_{2,\overline{q}}(\R^d,\gamma_d)}  \right \},\\
\| f \|_1 &:= \max \left \{ \| f \|_{\hoeldO {\theta,\overline{q}}}, \| f\|_{\B_{2,\underline{q}}(\R^d,\gamma_d)} \right \}.
\end{align*}
Then one has that 
$\| \cdot \|_0 \le c \| \cdot \|_1$ for some $c>0$
because of \eqref{eqn:embedding_interpolation_same_theta} (which is a continuous embedding). If we find $f_N$ such that
\[ \sup_{N=1,2,\ldots} \| f_N  \|_0 < \infty \sptext{1}{and}{1}
   \sup_{N=1,2,\ldots} \| f_N \|_1 =\infty, \]
then, by the open mapping theorem, there is an $f$ such that $\| f  \|_1 = \infty$ and
$\| f  \|_0 <\infty$. 
This would imply $f\in \hoeldO{\theta,\overline{q}}$ but $f\not\in \B_{2,\underline{q}}(\R^d,\gamma_d)$ and therefore
the theorem would be verified.
\medskip

(b) To construct such $f_N$ we first choose a non-constant $B:\R^d\to [0,\kappa]\in C^\infty$ 
with $|B|_{\Lip}\le \kappa$ for some $\kappa>0$ and with $\supp(B)\subseteq [-1,1]^d$.
From $B$ we construct 
\[ b(x):= \sum_{k_1,\ldots,k_d=-\infty}^\infty 1_{(2k_1,\ldots,2k_d)+[-1,1]^d}(x) B\big (x-(2k_1,\ldots,2k_d)\big ).\] 
Using \cref{statement:periodicity_implies_oscillation} (with $R=\sqrt{d}$)
and \cref{statement:scaling} we construct
$(b_n)_{n=1}^\infty$ satisfying \eqref{eqn:extreme_1} and \eqref{eqn:extreme_2}. 
Then we take an $a=(\alpha_n)_{n\ge 1} \in \ell_{\overline{q}}^{(\theta)} \setminus  \ell_{\underline{q}}^{(\theta)}$
and let 
$a^N := (\alpha_1,\ldots,\alpha_N,0,0,\ldots)$.
From \cref{statement:random_basis_of_extreme_elements} (applied to $(\theta,\underline{q})$) it follows that 
for each $N$ there are signs $\vare_1^N,\ldots,\vare_N^N\in \{ -1,1 \}$ such that
for $f_N := \sum_{n=1}^N \vare_n^N \alpha_n b_n$ one has
\[
      \frac{1}{\coref{eqn:statement:random_basis_of_extreme_elements}^{(\underline{q})}} 
      \left \| a^N \right \|_{\ell_{\underline{q}}^{(\theta)}} 
  \le \int_M \left \|  \sum_{n=1}^N \alpha_n \vare_n(\xi) b_n  \right \|_{B^\theta_{2,\underline{q}}(\R^d,\gamma_d)} \rho(\od \xi)  
  \le  \left \| f_N  \right \|_{B^\theta_{2,\underline{q}}(\R^d,\gamma_d)}
\]
with $\| a^N \|_{\ell_{\underline{q}}^{(\theta)}}\to \infty$ as $N\to \infty$.
Therefore, $\sup_{N=1,2,\ldots} \| f_N\|_1 =\infty$.
\medskip

(c) Using  \cref{statement:random_basis_of_extreme_elements} (this time applied to $(\theta,\overline{q})$)
gives 
\[     \| f_N \|_{\B_{2,\overline{q}}^\theta(\R^d,\gamma_d)}
   \le \coref{eqn:statement:random_basis_of_extreme_elements}^{(\overline{q})}  \| a^N \|_{\ell_{\overline{q}}^{(\theta)}} 
   \le \coref{eqn:statement:random_basis_of_extreme_elements}^{(\overline{q})}  \| a   \|_{\ell_{\overline{q}}^{(\theta)}} < \infty. \]

(d) Finally, we show that there is a constant $c>0$ such that 
\[     \| f_N \|_{\hoeldO {\theta,\overline{q}}}
   \le c \| a^N \|_{\ell_{\overline{q}}^{(\theta)}} 
   \le c \| a   \|_{\ell_{\overline{q}}^{(\theta)}} < \infty. \]
For this purpose we define a linear operator bounded as operator between
\[ T: \ell_1^{(0)} \to C_b^0(\R^d) 
   \sptext{1}{and}{1}
   T: \ell_1^{(1)} \to \Lip^0(\R^d) \]
by $T((\beta_n)_{n=1}^\infty):= \sum_{n=1}^\infty \beta_n b_n$. By interpolation we get that 
\[  T: \ell_{\overline{q}}^{(\theta)} \to (C_b^0(\R^d),\Lip^0(\R^d))_{\theta,\overline{q}} \]
is bounded as well. 
\medskip

(e) Combining (b), (c), and (d) gives
\[ \sup_{N=1,2,\ldots}\| f_N\|_0 < \infty 
   \sptext{1}{and}{1}
   \sup_{N=1,2,\ldots}\| f_N\|_1 = \infty \]
and the proof is complete. 
\end{proof}


\section{An application to parabolic PDEs}

We use $d=1$ and a stochastic basis and Brownian motion 
$(\Omega,\cF,\P,(\cF_t)_{t\in [0,1]})$ and $(W_t)_{t\in [0,1]}$ as in \cref{sec:description_K-functional},
here restricted to $t\in [0,1]$ and with $\cF=\cF_1$.
For $f\in L_2(\R,\gamma_1)$ we let
$F:[0,T]\times \R\to \R$ be given by
\[ F(t,x) := \E f(x+W_{1-t}). \]
Moreover, for $t\in [0,1)$ and $\alpha>0$ we let
\[ \varphi_t := \frac{\partial F}{\partial x}(t,W_t) 
   \sptext{1}{and}{1} 
   \cI^\alpha_t \varphi := \alpha \int_0^1 (1-u)^{\alpha-1} \varphi_{u\wedge t} \od u. \] 
For $\theta \in (0,1)$ we know from \cite[Theorem 6.6]{Geiss:Nguyen:20} that 
\begin{equation}\label{eqn:BMO_vs_Hoelder_with_RL}
       \cI^\alpha \varphi - \varphi_0 \in \BMO_2([0,1))
       \sptext{1}{for}{1} f \in \hoelO{\theta,2}
       \sptext{1}{and}{1} \alpha:=\frac{1-\theta}{2}.
\end{equation}
The question is whether we can use the condition $f\in \hoelO{\theta,q}$ for some
$q\in (2,\infty]$. We will disprove this:

\begin{theo}
If $(\theta,q)\in (0,1)\times (2,\infty]$, then there is an $f\in \hoelO {\theta,q}$ such that 
\[ \cI^\alpha \varphi - \varphi_0 \not \in \BMO_2([0,1))
   \sptext{1}{for}{1} \alpha:=\frac{1-\theta}{2} . \]
\end{theo}

\begin{proof}
Assume $f\in \hoelO {\theta,q}$. From \cite[formula (3.3) of Proposition 3.8]{Geiss:Nguyen:20} one gets
\[ \cI^\alpha_t \varphi = \varphi_0 + \int_{(0,t]} (1-u)^\alpha \od \varphi_u
   \quad \mbox{a.s.} \] 
If $ \cI^\alpha \varphi - \varphi_0 \in \BMO_2([0,1))$ would hold, then
It\^o's isometry would give 
\begin{align*}
     \sqrt{\E \int_0^1 (1-u)^{1-\theta} \left | \frac{\partial^2 F}{\partial x^2} (u,W_u) \right |^2 \od u} 
& =  \sqrt{\E \int_0^1 (1-u)^{2\alpha} \left | \frac{\partial^2 F}{\partial x^2} (u,W_u) \right |^2 \od u} \\
& =  \sqrt{\E \int_0^1 (1-u)^{2\alpha} \od [\varphi]_u} \\
& =  \sup_{t\in [0,1)}\|\cI_t^\alpha \varphi - \varphi_0\|_{L_2} \\
&\le \| \cI^\alpha \varphi - \varphi_0 \|_{\BMO_2([0,1))}
  < \infty.
\end{align*}
By \cite[Theorem 3.1]{Geiss:Toivola:15} this would imply $f \in \B_{2,2}^\theta(\R,\gamma_1)$.
Now \cref{statement:sharp_embedding_Hoelder_Besov} 
applied to $(\underline{q},\overline{q})=(2,q)$ implies our statement.
\end{proof}


\bibliographystyle{amsplain}

\end{document}